\documentclass[a4paper,12pt,reqno]{amsart}
\usepackage{setspace} 
\usepackage{amsmath, amsfonts,amsthm,amssymb,mathrsfs,cases}
\usepackage{enumerate}
\usepackage{graphicx}
\usepackage{ascmac}
\usepackage[arrow,matrix]{xy}
\usepackage{color}
\usepackage{array}
\usepackage{pifont}
\usepackage{longtable}
\usepackage{mathtools}

\theoremstyle{definition}
\newtheorem{theorem}{Theorem}[section]
\newtheorem*{theorem*}{Theorem}

\newtheorem*{definition*}{Definition}

\newtheorem*{lemma*}{Lemma}

\newtheorem*{example*}{Example}
\newtheorem{proposition}[theorem]{Proposition}
\newtheorem*{proposition*}{Proposition}

\newtheorem*{corollary*}{Corollary}
\newtheorem{remark}[theorem]{Remark}
\newtheorem*{remark*}{Remark}

\newtheorem{maintheorem}{Theorem}

\newcommand{\ctext}[1]{\raise0.2ex\hbox{\textcircled{\scriptsize{#1}}}}

\DeclareMathOperator{\Spin}{Spin}
\DeclareMathOperator{\Pin}{Pin}

\DeclareMathOperator{\Id}{Id}
\DeclareMathOperator{\tr}{tr}

\DeclareMathOperator{\SO}{SO}

\DeclareMathOperator{\G}{G}

\DeclareMathOperator{\SU}{SU}

\DeclareMathOperator{\Sp}{Sp}

\DeclareMathOperator{\ad}{ad}

\title{Left-invariant Rarita-Schwinger fields on 3-dimensional Lie groups}

\author{Soma Ohno}

\address{Soma Ohno, Faculty of Human Informatics, Aichi Shukutoku University, 2-9 Katahira, Nagakute, Aichi, 480-1197, Japan.}
\email{sohno@asu.aasa.ac.jp}

\begin{document}

\begin{abstract}
We classify left-invariant Rarita-Schwinger fields on 2- and 3-dimensional Lie groups with left-invariant Riemannian metrics. The left-invariant Rarita-Schwinger equation is reduced to an algebraic system on the associated metric Lie algebra. We show that, in dimension 2, non-trivial examples occur only in the flat abelian case. In dimension 3, apart from the flat abelian case, the only non-trivial examples occur on $\SU(2)$ equipped with a special left-invariant metric, for which the space of left-invariant Rarita-Schwinger fields has complex dimension 2. This metric is a Berger metric on $S^3$.
\end{abstract}

\maketitle

\makeatletter
  \renewcommand{\theequation}{
  \thesection.\arabic{equation}}
 \@addtoreset{equation}{section}
\makeatother

\section{Introduction}
The Rarita-Schwinger equation was introduced by W. Rarita and J. Schwinger \cite{RaritaSchwinger} as a relativistic wave equation for particles of spin-$3/2$. Since then, it has played an important role in mathematical physics, especially in the study of higher spin fields and gravitinos. From the viewpoint of spin geometry, the Rarita-Schwinger operator arises as the spin-$3/2$ component of the twisted Dirac operator and is a generalization of the classical Dirac operator. Rarita-Schwinger fields are spin-$3/2$ fields satisfying the Rarita-Schwinger equation. 

In Riemannian spin geometry, the Rarita-Schwinger operator and Rarita-Schwinger fields have been studied in relation to special geometric structures. Y. Homma and U. Semmelmann \cite{HommaSemmelmann} investigated the kernel of the Rarita-Schwinger operator on compact Riemannian spin manifolds and obtained results for several important classes of manifolds, including quaternionic K\"{a}hler manifolds, symmetric spaces, and manifolds with special holonomy. C. B\"{a}r and R. Mazzeo \cite{BarMazzeo} constructed compact Riemannian manifolds with many Rarita-Schwinger fields. Rarita-Schwinger fields have also been studied on compact strict nearly K\"{a}hler 6-manifolds and on compact nearly parallel $\G_2$-manifolds (\cite{Ohno}, \cite{OhnoTomihisa}). The Rarita-Schwinger operator also appears in gauge theory; for instance, Nguyen \cite{Nguyen} introduced a Seiberg-Witten type equation in which the Dirac operator is replaced by the Rarita-Schwinger operator. These works indicate that Rarita-Schwinger fields are closely connected with Killing spinors, parallel spinors, special holonomy geometry, and gauge theory.

The purpose of the present paper is to study Rarita-Schwinger fields from a different and more explicit point of view, namely, the left-invariant geometry of Lie groups. Lie groups endowed with left-invariant metrics provide a rich class of homogeneous Riemannian spin manifolds. One advantage of this setting is that differential equations for invariant fields can often be reduced to algebraic equations on the corresponding metric Lie algebra. This idea has been used effectively in the study of Dirac operators and harmonic spinors on Lie groups; see, for example, \cite{AmmannBar}, \cite{BazzoniLuciaVicente}, and \cite{GilGarciaRusso}. 

Our main result is a complete classification of left-invariant Rarita-Schwinger fields on 2- and 3-dimensional Lie groups endowed with left-invariant Riemannian metrics. Since connected and simply connected Lie groups are determined, up to isomorphism, by their Lie algebras, the 3-dimensional case can be studied at the Lie algebra level. The classification of 3-dimensional real Lie algebras goes back to Bianchi \cite{Bianchi1,Bianchi2}; see also Milnor's work on left-invariant metrics on Lie groups \cite{Milnor}. We use this low-dimensional Lie-theoretic framework to reduce the Rarita-Schwinger equation to explicit algebraic systems. We now state the main results of this paper.

\begin{maintheorem} \label{thmA}
Let $G$ be a connected and simply connected 2-dimensional Lie group endowed with a left-invariant Riemannian metric. If $G$ is abelian, then every left-invariant spin-$3/2$ spinor is parallel, and hence gives a Rarita-Schwinger field. If $G$ is non-abelian, then there are no non-trivial left-invariant Rarita-Schwinger fields on $G$.
\end{maintheorem}

\begin{maintheorem} \label{thmB}
Let $G$ be a connected and simply connected 3-dimensional Lie group endowed with a left-invariant Riemannian metric. If $G$ is abelian, then every left-invariant spin-$3/2$ spinor is parallel, and hence gives a Rarita-Schwinger field. If $G$ is non-abelian, then there are non-trivial left-invariant Rarita-Schwinger fields only on $G\cong\SU(2)$ equipped with a special left-invariant metric. For this metric, the space of left-invariant Rarita-Schwinger fields has complex dimension 2.
\end{maintheorem}

Theorem \ref{thmA} shows that, in dimension 2, the only left-invariant examples come from the flat abelian case. Theorem \ref{thmB} shows that the 3-dimensional situation is more subtle but still very rigid. In the unimodular case, we examine the Heisenberg Lie algebra $\mathfrak{heis}_3$, the Euclidean motion Lie algebra $\mathfrak{e}(2)$, the Lorentzian motion Lie algebra $\mathfrak{e}(1,1)$, $\mathfrak{sl}(2,\mathbb{R})$, and $\mathfrak{su}(2)$. Among the corresponding simply connected Lie groups, non-trivial left-invariant Rarita-Schwinger fields outside the flat abelian case occur only for the exceptional metrics on $\SU(2)$. In the non-unimodular case, using standard normal forms for metric Lie algebras, we show that the Rarita-Schwinger equation does not admit non-trivial left-invariant solutions. 

The special left-invariant metric on $\SU(2)$ appearing in Theorem \ref{thmB} will be described explicitly in Subsection \ref{subsec: Berger sphere}. It is a Berger metric on $S^3$. Berger spheres have played an important role in spin geometry, notably in the study of Dirac spectra and harmonic spinors \cite{Bar}, \cite{Hitchin}. Our result shows that a special Berger sphere admits non-trivial left-invariant Rarita-Schwinger fields, although the corresponding metric is not Einstein. This gives an explicit compact homogeneous example of a non-Einstein metric carrying left-invariant Rarita-Schwinger fields. Moreover, the round $3$-sphere admits no Rarita-Schwinger fields, whereas the Berger metric obtained here admits non-trivial left-invariant Rarita-Schwinger fields. Therefore, this gives a second example in which the existence of Rarita-Schwinger fields on the same manifold depends on the choice of metric. The first example in which the existence depends on the metric was obtained on $S^3\times S^3$ in \cite{OhnoTomihisa}.

This paper is organized as follows. In Section \ref{Preliminaries}, we recall the definition of Rarita-Schwinger fields, introduce left-invariant spinors and vector spinors on Lie groups, and derive a formula for the twisted Dirac operator acting on left-invariant vector spinors. In Section \ref{Left-invariant Rarita-Schwinger fields on low dimensional Lie groups}, we classify left-invariant Rarita-Schwinger fields on low-dimensional Lie groups. We first treat the 2-dimensional case. We then compute the component expression of the twisted Dirac operator for an arbitrary $3$-dimensional metric Lie algebra and derive the left-invariant Rarita-Schwinger equation under the spin-$3/2$ condition. Using this general formula, we classify the $3$-dimensional unimodular and non-unimodular cases separately. Finally, we describe the exceptional metric on $\mathfrak{su}(2)$ as a Berger metric on $\SU(2)\cong S^3$.

\section{Preliminaries} \label{Preliminaries}
\subsection{Rarita-Schwinger fields}
We first recall the basic definitions of vector spinors, spin-$3/2$ spinors, and Rarita-Schwinger fields, following the standard formulation used in \cite{HommaSemmelmann}.

Let $(M^n,g)$ be an $n$-dimensional Riemannian spin manifold, and let $\Sigma_{1/2}M$ denote its complex spinor bundle. We also denote by $TM^{\mathbb C}$ the complexified tangent bundle. The tensor product $\Sigma_{1/2}M\otimes TM^{\mathbb C}$ is called the vector spinor bundle. There is a natural Clifford contraction $\Pi:\Sigma_{1/2}M\otimes TM^{\mathbb C}\to \Sigma_{1/2}M$ defined locally by
\[
\Pi\left(\sum_{i=1}^n \psi_i\otimes e_i\right)=\sum_{i=1}^n e_i\cdot\psi_i,
\]
where $(e_1,\ldots,e_n)$ is a local orthonormal frame. The spin-$3/2$ spinor bundle is defined as the kernel of this contraction: $\Sigma_{3/2}M\coloneqq \ker\Pi \subset \Sigma_{1/2}M\otimes TM^{\mathbb{C}}$. Thus the vector spinor bundle decomposes, as a $\Spin(n)$-module, into
\[
\Sigma_{1/2}M \otimes TM^{\mathbb C} \cong \Sigma_{1/2}M \oplus \Sigma_{3/2}M.
\]
Let $\nabla$ be the connection on $\Sigma_{1/2}M\otimes TM^{\mathbb C}$ induced by the spin connection on $\Sigma_{1/2}M$ and the Levi-Civita connection on $TM^{\mathbb C}$. The twisted Dirac operator on vector spinors is defined by
\[
D_{TM}=\sum_{k=1}^{n}(e_k\cdot\otimes \Id_{TM^{\mathbb C}})\circ \nabla_{e_k}.
\]
With respect to the above decomposition $\Sigma_{1/2}M\otimes TM^{\mathbb C}\cong\Sigma_{1/2}M\oplus \Sigma_{3/2}M$, the twisted Dirac operator $D_{TM}$ can be written as a $2\times 2$ matrix of first-order differential operators:
\[
D_{TM}=\left(
\begin{array}{cc}
\frac{2-n}{n}D & 2P^{\ast} \\[1ex]
\frac{2}{n}P & Q
\end{array}
\right).
\]
Here $D:\Gamma(\Sigma_{1/2}M)\to \Gamma(\Sigma_{1/2}M)$ is the ordinary Dirac operator, $P:\Gamma(\Sigma_{1/2}M)\to \Gamma(\Sigma_{3/2}M)$ is the Penrose operator, and $P^\ast$ denotes its formal adjoint. The lower right component $Q:\Gamma(\Sigma_{3/2}M)\to \Gamma(\Sigma_{3/2}M)$ is called the Rarita-Schwinger operator. It is a formally self-adjoint first-order elliptic differential operator.

A Rarita-Schwinger field is a section $\Psi\in\Gamma(\Sigma_{1/2}M\otimes TM^{\mathbb C})$ satisfying
\[
\Psi\in\Gamma(\Sigma_{3/2}M), \qquad P^{\ast}\Psi=0,\qquad Q\Psi=0.
\]
Equivalently, a Rarita-Schwinger field is a spin-$3/2$ spinor $\Psi\in\Gamma(\Sigma_{3/2}M)$ satisfying $D_{TM}\Psi=0$. 

\subsection{The twisted Dirac operator on left-invariant vector spinors}
In this subsection, we adapt the above definitions to Lie groups endowed with left-invariant Riemannian metrics and describe the twisted Dirac operator on left-invariant vector spinors. A left-invariant orthonormal frame trivializes the tangent bundle, and hence the spinor and vector spinor bundles. Thus the Rarita-Schwinger equation for left-invariant spin-$3/2$ spinors can be written as an algebraic equation on the corresponding metric Lie algebra. This point of view is standard in the study of left-invariant harmonic spinors on Lie groups; see, for example, \cite{GilGarciaRusso}.

Let $G$ be a connected and simply connected $n$-dimensional Lie group endowed with a left-invariant Riemannian metric $g$. We denote by $\mathfrak{g}$ its Lie algebra. Choose an oriented orthonormal basis $(e_1,\ldots,e_n)$ of $(\mathfrak{g},g)$, and denote by the same symbols the corresponding left-invariant orthonormal frame on $G$. We identify $TG$ and $T^*G$ via the metric and use the same symbols $(e_1,\ldots,e_n)$ for the orthonormal frame and its dual coframe. In particular, $de_i$ denotes the exterior derivative of the corresponding left-invariant $1$-form. For a left-invariant $1$-form $X$ and left-invariant vector fields $Y,Z$, the Maurer-Cartan formula reads
\begin{equation}\label{eq: maurer-cartan}
dX(Y,Z)=-X([Y,Z])=-\langle X,[Y,Z]\rangle.
\end{equation}
By left translations, the tangent bundle is trivialized as $TG\cong G\times \mathfrak{g}$. Under this trivialization, the oriented orthonormal frame bundle is identified with $\mathbf{SO}(G)\cong G\times \SO(n)$. Since $G$ is simply connected, it has a unique spin structure. With respect to the above left-invariant frame, the principal $\Spin(n)$-bundle may be identified with $\mathbf{Spin}(G)\cong G\times \Spin(n)$. Consequently, the spinor bundle is trivialized as $\Sigma_{1/2}G\cong G\times \Delta_n$, where $\Delta_n$ denotes the complex spin representation of $\Spin(n)$. A spinor field on \(G\) is called left-invariant if it is constant with respect to this trivialization. Thus the space of left-invariant spinors is naturally identified with $\Delta_n$.

Similarly, the vector spinor bundle is trivialized as $\Sigma_{1/2}G\otimes TG^{\mathbb C} \cong G\times(\Delta_n\otimes\mathfrak g^{\mathbb C})$. We call a vector spinor field left-invariant if it is constant with respect to this trivialization. Hence a left-invariant vector spinor is identified with an element of $\Delta_n\otimes\mathfrak g^{\mathbb C}$, and can be written uniquely as
\[
\Psi=\sum_{i=1}^{n}\psi_i\otimes e_i,
\qquad
\psi_i\in\Delta_n.
\]

A left-invariant spin-$3/2$ spinor is therefore a left-invariant vector spinor $\displaystyle \Psi=\sum_{i=1}^n\psi_i\otimes e_i$ satisfying $\displaystyle \sum_{i=1}^n e_i\cdot\psi_i=0$.

We now describe the twisted Dirac operator on left-invariant vector spinors.
\begin{proposition} \label{prop: twisted-dirac-left-invariant}
Let $G$ be a connected and simply connected $n$-dimensional Lie group endowed with a left-invariant Riemannian metric, let $(e_1,\ldots,e_n)$ be a left-invariant orthonormal frame, and let $\displaystyle \Psi=\sum_{k=1}^n\psi_k\otimes e_k$ be a left-invariant vector spinor. Then
\begin{align*}
-4D_{TM}\Psi=&\sum_{i,k=1}^n\left(e_i\wedge de_i+2e_i\mathbin{\lrcorner}de_i\right)\cdot\psi_k\otimes e_k -2\sum_{j,k=1}^n(e_k\mathbin{\lrcorner}de_j)\cdot\psi_k\otimes e_j \\
&+2\sum_{j,k=1}^n(e_j\mathbin{\lrcorner}de_k)\cdot\psi_k\otimes e_j -2\sum_{i,j,k=1}^nde_i(e_k,e_j)e_i\cdot\psi_k\otimes e_j .
\end{align*}
\end{proposition}

\begin{proof}
By the definition of the twisted Dirac operator,
\begin{align*}
-4D_{TM}\Psi=&-4\sum_{i,k=1}^n(e_i\cdot\otimes \Id)\nabla_{e_i}(\psi_k\otimes e_k) \\
=&-4\sum_{k=1}^n D\psi_k\otimes e_k-4\sum_{i,k=1}^ne_i\cdot\psi_k\otimes\nabla_{e_i}e_k.
\end{align*}
For left-invariant vector fields $X,Y$, the Levi--Civita connection is given by
\[
\nabla_XY=\frac{1}{2}\sum_{j=1}^n\left(-de_j(X,Y)+dY(X,e_j)+dX(Y,e_j)\right)e_j.
\]
%Applying this formula to $X=e_i$ and $Y=e_k$, we obtain
%\[
%\nabla_{e_i}e_k
%\frac{1}{2}\sum_{j=1}^n\left(-de_j(e_i,e_k)+de_k(e_i,e_j)+de_i(e_k,e_j)\right)e_j.
%\]
Therefore,
\begin{align*}
&-4\sum_{i,k=1}^n e_i\cdot\psi_k\otimes \nabla_{e_i}e_k \\
=&-2\sum_{i,j,k=1}^n\left(-de_j(e_i,e_k)+de_k(e_i,e_j)+de_i(e_k,e_j)\right)e_i\cdot\psi_k\otimes e_j  \\
=&-2\sum_{j,k=1}^n(e_k\mathbin{\lrcorner}de_j)\cdot\psi_k\otimes e_j+2\sum_{j,k=1}^n(e_j\mathbin{\lrcorner}de_k)\cdot\psi_k\otimes e_j-2\sum_{i,j,k=1}^n de_i(e_k,e_j)e_i\cdot\psi_k\otimes e_j .
\end{align*}
On the other hand, the formula for the ordinary Dirac operator on left-invariant spinors \cite[Proposition~1.1]{GilGarciaRusso} gives
\[
-4D\psi_k=\sum_{i=1}^n\left(e_i\wedge de_i+2e_i\mathbin{\lrcorner}de_i\right)\cdot\psi_k .
\]
Substituting these formulas into the expression for $-4D_{TM}\Psi$ yields the desired formula.
\end{proof}

\begin{remark}\label{rem: metric Lie algebra}
By Proposition \ref{prop: twisted-dirac-left-invariant}, the existence of left-invariant Rarita-Schwinger fields is determined by the structure equations of the corresponding metric Lie algebra $(\mathfrak g,\langle\cdot,\cdot\rangle)$. Strictly speaking, a left-invariant Rarita-Schwinger field is a section of the spin-$3/2$ spinor bundle over the Lie group $G$. For brevity, however, we will say that the corresponding metric Lie algebra $(\mathfrak g,\langle\cdot,\cdot\rangle)$ admits a left-invariant Rarita-Schwinger field.
\end{remark}

\begin{remark} \label{rem: abel lie algebra}
If $\mathfrak g$ is abelian, then $de_i=0$. It follows from Proposition \ref{prop: twisted-dirac-left-invariant} that every left-invariant vector spinor lies in the kernel of $D_{TM}$. In particular, every left-invariant spin-$3/2$ spinor is a Rarita-Schwinger field. Hence, in what follows, we consider only non-abelian Lie algebras.
\end{remark}

A Lie algebra $\mathfrak g$ is called unimodular if
\[
\tr(\ad_X)=0
\qquad
\text{for all }X\in\mathfrak{g}.
\]
A Lie algebra that is not unimodular is called non-unimodular. Moreover, a connected Lie group $G$ is unimodular if and only if its Lie algebra $\mathfrak g$ is unimodular.

\section{Left-invariant Rarita-Schwinger fields on low-dimensional Lie groups} \label{Left-invariant Rarita-Schwinger fields on low dimensional Lie groups}
\subsection{Left-invariant Rarita-Schwinger fields on 2-dimensional Lie groups}
Up to isomorphism, the only non-abelian $2$-dimensional real Lie algebra is the $2$-dimensional affine Lie algebra $\mathfrak{aff}(\mathbb R)$. For any inner product on $\mathfrak{aff}(\mathbb R)$, one can choose an orthonormal basis $(e_1,e_2)$ such that
\[
[e_1,e_2]=ae_2,
\]
where $a>0$. The Maurer-Cartan formula $(\ref{eq: maurer-cartan})$ gives
\[
de_1=0, \qquad de_2=-ae_1\wedge e_2.
\]
Therefore, applying Proposition \ref{prop: twisted-dirac-left-invariant} with respect to this orthonormal basis, we obtain the following expression for the action of the twisted Dirac operator on a left-invariant spin-$3/2$ spinor $\displaystyle \Psi=\sum_{i=1}^2\psi_i\otimes e_i$ on $\mathfrak{aff}(\mathbb R)$:
\begin{align*}
-4D_{TM}\Psi=&\sum_{i,k=1}^2(e_i\wedge de_i+2e_i\mathbin{\lrcorner}de_i)\cdot \psi_k\otimes e_k-2\sum_{j,k=1}^2(e_k\mathbin{\lrcorner}de_j)\cdot\psi_k\otimes e_j \\
&+2\sum_{j,k=1}^2(e_j\mathbin{\lrcorner}de_k)\cdot\psi_k\otimes e_j-2\sum_{i,j,k=1}^2de_i(e_k,e_j)e_i\cdot\psi_k\otimes e_j \\
=&(2ae_1\cdot\psi_1\otimes e_1+2ae_1\cdot\psi_2\otimes e_2)+2ae_2\cdot\psi_1\otimes e_2 \\
&-2ae_2\cdot\psi_2\otimes e_1+(2ae_2\cdot\psi_1\otimes e_2-2ae_2\cdot\psi_2\otimes e_1) \\
=&2a(e_1\cdot\psi_1-2e_2\cdot\psi_2)\otimes e_1+2a(e_1\cdot\psi_2+2e_2\cdot\psi_1)\otimes e_2.
\end{align*}
Using the spin-$3/2$ condition $\psi_2=e_2\cdot e_1\cdot\psi_1$, we further obtain
\[
-4D_{TM}\Psi=6ae_1\cdot\psi_1\otimes e_1+6ae_2\cdot\psi_1\otimes e_2.
\]
Hence $D_{TM}\Psi=0$ if and only if $\psi_1=0$. Thus, $\mathfrak{aff}(\mathbb R)$ admits no non-trivial left-invariant Rarita-Schwinger fields.

Combining this with the abelian case discussed in Remark \ref{rem: abel lie algebra}, we obtain the first main theorem.
\let\temp\thetheorem
\renewcommand{\thetheorem}{\ref{thmA}}

\begin{theorem}
The only connected and simply connected $2$-dimensional Lie group admitting a non-trivial left-invariant Rarita-Schwinger field is the abelian Lie group $\mathbb R^2$.
\end{theorem}

\let\thetheorem\temp
\addtocounter{theorem}{-1}

\subsection{The left-invariant Rarita-Schwinger equation on $3$-dimensional Lie groups} \label{Left-invariant Rarita-Schwinger fields on 3-dimensional Lie groups}
Let $(e_1,e_2,e_3)$ be an orthonormal basis of a $3$-dimensional metric Lie algebra $(\mathfrak g,\langle\cdot,\cdot\rangle)$, and write
\begin{equation} \label{eq: exterior derivative on 3-dim lie alg}
de_1=\sum_{i<j}a_{ij}e_i \wedge e_j, \quad de_2=\sum_{i<j}b_{ij}e_i \wedge e_j, \quad de_3=\sum_{i<j}c_{ij}e_i \wedge e_j,
\end{equation}
where $a_{ij},b_{ij},c_{ij}\in\mathbb R$. Using this notation and Proposition~\ref{prop: twisted-dirac-left-invariant}, the action of the twisted Dirac operator on a left-invariant vector spinor $\displaystyle \Psi=\sum_{i=1}^3\psi_i\otimes e_i$ can be written as
\begin{equation}
-4D_{TM}\Psi=-4(D_{TM}\Psi)_1\otimes e_1-4(D_{TM}\Psi)_2\otimes e_2-4(D_{TM}\Psi)_3\otimes e_3,
\end{equation}
where the components are given by
\begin{align*}
-4(D_{TM}\Psi)_1=&(a_{23}-b_{13}+c_{12})(e_1\wedge e_2\wedge e_3)\cdot\psi_1 \\
&+2(-b_{12}e_1-c_{13}e_1+a_{12}e_2-c_{23}e_2+a_{13}e_3+b_{23}e_3)\cdot\psi_1 \\
&+2(2a_{12}e_1+2b_{12}e_2-a_{23}e_3+b_{13}e_3+c_{12}e_3)\cdot\psi_2 \\
&+2(2a_{13}e_1+a_{23}e_2+b_{13}e_2+c_{12}e_2+2c_{13}e_3)\cdot\psi_3, \\
-4(D_{TM}\Psi)_2=&(a_{23}-b_{13}+c_{12})(e_1\wedge e_2\wedge e_3)\cdot\psi_2 \\
&+2(-2a_{12}e_1-2b_{12}e_2+a_{23}e_3-b_{13}e_3-c_{12}e_3)\cdot\psi_1 \\
&+2(-b_{12}e_1-c_{13}e_1+a_{12}e_2-c_{23}e_2+a_{13}e_3+b_{23}e_3)\cdot\psi_2 \\
&+2(a_{23}e_1+b_{13}e_1-c_{12}e_1+2b_{23}e_2+2c_{23}e_3)\cdot\psi_3, \\
-4(D_{TM}\Psi)_3=&(a_{23}-b_{13}+c_{12})(e_1\wedge e_2\wedge e_3)\cdot\psi_3 \\
&+2(-2a_{13}e_1-a_{23}e_2-b_{13}e_2-c_{12}e_2-2c_{13}e_3)\cdot\psi_1 \\
&+2(-a_{23}e_1-b_{13}e_1+c_{12}e_1-2b_{23}e_2-2c_{23}e_3)\cdot\psi_2 \\
&+2(-b_{12}e_1-c_{13}e_1+a_{12}e_2-c_{23}e_2+a_{13}e_3+b_{23}e_3)\cdot\psi_3.
\end{align*}
Using the spin-$3/2$ condition $\psi_3=e_3\cdot e_1\cdot\psi_1+e_3\cdot e_2\cdot\psi_2$, we can rewrite $-4(D_{TM}\Psi)_1$, $-4(D_{TM}\Psi)_2$, and $-4(D_{TM}\Psi)_3$ as follows:
\begin{equation} \label{eq: twisted Dirac for left-inv spin-3/2 on 3-dim}
\begin{split}
-4(D_{TM}\Psi)_1=&(3a_{23}+b_{13}+3c_{12})(e_1\wedge e_2\wedge e_3)\cdot\psi_1 \\
&+2(-b_{12}e_1-3c_{13}e_1+a_{12}e_2-c_{23}e_2+3a_{13}e_3+b_{23}e_3)\cdot\psi_1 \\
&-4a_{13}(e_1\wedge e_2\wedge e_3)\cdot\psi_2 \\
&+2(2a_{12}e_1+2b_{12}e_2-2c_{13}e_2+2b_{13}e_3+2c_{12}e_3)\cdot\psi_2, \\
-4(D_{TM}\Psi)_2=& 4b_{23}(e_1\wedge e_2\wedge e_3)\cdot\psi_1 \\
&+2(-2a_{12}e_1-2c_{23}e_1-2b_{12}e_2+2a_{23}e_3-2c_{12}e_3)\cdot\psi_1 \\
&+(-a_{23}-3b_{13}+3c_{12})(e_1\wedge e_2\wedge e_3)\cdot\psi_2 \\
&+2(-b_{12}e_1-c_{13}e_1+a_{12}e_2-3c_{23}e_2+a_{13}e_3+3b_{23}e_3)\cdot\psi_2, \\
-4(D_{TM}\Psi)_3=& 2(a_{12}-c_{23})(e_1\wedge e_2\wedge e_3)\cdot\psi_1 \\
&+(-6a_{13}e_1-2b_{23}e_1-3a_{23}e_2-b_{13}e_2-3c_{12}e_2-2b_{12}e_3-6c_{13}e_3)\cdot\psi_1 \\
&+2(b_{12}+c_{13})(e_1\wedge e_2\wedge e_3)\cdot\psi_2 \\
&+(-a_{23}e_1-3b_{13}e_1+3c_{12}e_1-2a_{13}e_2-6b_{23}e_2+2a_{12}e_3-6c_{23}e_3)\cdot\psi_2.
\end{split}
\end{equation}

From now on, we identify the $3$-dimensional complex spin representation with $\Delta_3\cong\mathbb C^2$ and use the following Pauli-matrix realization of Clifford multiplication with respect to a positively oriented orthonormal basis $(e_1,e_2,e_3)$:
\[
e_1\cdot=
\begin{pmatrix}
0&i\\
i&0
\end{pmatrix},
\qquad
e_2\cdot=
\begin{pmatrix}
0&1\\
-1&0
\end{pmatrix},
\qquad
e_3\cdot=
\begin{pmatrix}
i&0\\
0&-i
\end{pmatrix}.
\]
Then $(e_1\wedge e_2\wedge e_3)\cdot=I$. Moreover, for $x,y\in\mathbb R$,
\[
xI+ye_3\cdot=
\begin{pmatrix}
x+iy&0\\
0&x-iy
\end{pmatrix},
\]
and therefore
\begin{equation} \label{eq: determinant of xI plus ye3}
\det(xI+ye_3\cdot)=x^2+y^2.
\end{equation}

\subsection{Left-invariant Rarita-Schwinger fields on 3-dimensional unimodular Lie groups} \label{Left-invariant Rarita-Schwinger fields on 3-dimensional unimodular Lie groups}
We determine which $3$-dimensional unimodular metric Lie algebras admit non-trivial left-invariant Rarita-Schwinger fields. We first recall Milnor's normal form for $3$-dimensional unimodular metric Lie algebras.

\begin{theorem}[\cite{Milnor}, Lemma 4.1]
Let $\mathfrak g$ be a $3$-dimensional unimodular Lie algebra. Then, for every inner product on $\mathfrak g$, there exist an orthonormal basis $(e_1,e_2,e_3)$ and real numbers
$\alpha,\beta,\gamma \in \mathbb{R}$ satisfying
\begin{equation} \label{eq: Milnor frame}
[e_1,e_2]=\alpha e_3,\qquad [e_2,e_3]=\beta e_1,\qquad [e_3,e_1]=\gamma e_2.
\end{equation}
\end{theorem}

This theorem yields the classification of $3$-dimensional unimodular real Lie algebras. After suitably permuting the basis and choosing its orientation, the sign patterns of $(\alpha,\beta,\gamma)$ correspond as follows: $(+++)$ to $\mathfrak{su}(2)$, $(++-)$ to $\mathfrak{sl}(2,\mathbb R)$, $(++0)$ to the Euclidean motion Lie algebra $\mathfrak e(2)$, $(+-0)$ to the Lorentzian motion Lie algebra $\mathfrak e(1,1)$, $(+00)$ to the Heisenberg Lie algebra $\mathfrak{heis}_3$, and $(000)$ to the abelian Lie algebra $\mathbb R^3$.

For a Milnor frame $(e_1,e_2,e_3)$ satisfying $(\ref{eq: Milnor frame})$, the Maurer-Cartan formula $(\ref{eq: maurer-cartan})$ gives
\[
de_1=-\beta e_2\wedge e_3,\qquad de_2=\gamma e_1\wedge e_3,\qquad de_3=-\alpha e_1\wedge e_2.
\]
Thus, in the notation of $(\ref{eq: exterior derivative on 3-dim lie alg})$,
\[
a_{23}=-\beta, \qquad b_{13}=\gamma, \qquad c_{12}=-\alpha,
\]
and all other components vanish. Substituting these expressions into $(\ref{eq: twisted Dirac for left-inv spin-3/2 on 3-dim})$, we obtain
\begin{align*}
-4(D_{TM}\Psi)_1=&(-3\alpha-3\beta+\gamma)\psi_1+4(-\alpha+\gamma)e_3\cdot\psi_2, \\
-4(D_{TM}\Psi)_2=&4(\alpha-\beta)e_3\cdot\psi_1+(-3\alpha+\beta-3\gamma)\psi_2, \\
-4(D_{TM}\Psi)_3=&(3\alpha+3\beta-\gamma)e_2\cdot\psi_1+(-3\alpha+\beta-3\gamma)e_1\cdot\psi_2.
\end{align*}
Setting $r=-3\alpha-3\beta+\gamma$, $s=-3\alpha+\beta-3\gamma$, $u=-\alpha+\gamma$, and $v=\alpha-\beta$, the left-invariant Rarita-Schwinger equation reduces to
\begin{equation}\label{eq: unimodular RS system}
\begin{cases}
r\psi_1+4ue_3\cdot\psi_2=0,\\
4ve_3\cdot\psi_1+s\psi_2=0,\\
-re_2\cdot\psi_1+se_1\cdot\psi_2=0.
\end{cases}
\end{equation}

\subsubsection{$\mathfrak{heis}_3$}
For $\mathfrak{heis}_3$, we have
\[
\alpha>0,\qquad \beta=\gamma=0.
\]
Hence
\[
r=s=-3\alpha,\qquad u=-\alpha,\qquad v=\alpha,
\]
and the left-invariant Rarita-Schwinger equation $(\ref{eq: unimodular RS system})$ becomes
\begin{equation}\label{eq: unimodular RS system on heis}
\begin{cases}
-3\alpha\psi_1-4\alpha e_3\cdot\psi_2=0,\\
4\alpha e_3\cdot\psi_1-3\alpha \psi_2=0,\\
3\alpha e_2\cdot\psi_1-3\alpha e_1\cdot\psi_2=0
\end{cases}
\end{equation}
The first equation of $(\ref{eq: unimodular RS system on heis})$ gives
\[
\psi_1=-\dfrac{4}{3}e_3\cdot\psi_2.
\]
Substituting this into the second equation of $(\ref{eq: unimodular RS system on heis})$ yields
\[
4e_3\cdot\left(-\dfrac{4}{3}e_3\cdot\psi_2\right)-3\psi_2=0.
\]
After simplification, this gives $\psi_2=0$. It follows immediately that $\psi_1=0$. Therefore, no non-trivial solution exists.

\subsubsection{$\mathfrak{e}(2)$}
For $\mathfrak e(2)$, we have
\[
\alpha>0,\qquad \beta>0, \qquad \gamma=0.
\]
Hence
\[
r=-3\alpha-3\beta,\qquad s=-3\alpha+\beta,\qquad u=-\alpha,\qquad v=\alpha-\beta.
\]
Since $\alpha>0$ and $\beta>0$, we have $r=-3\alpha-3\beta<0$, and in particular $r\neq0$. Thus, the first equation of $(\ref{eq: unimodular RS system})$ can be rewritten as
\[
\psi_1=-\dfrac{4u}{r}e_3\cdot\psi_2.
\]
Substitution into the third equation of $(\ref{eq: unimodular RS system})$ gives
\[
4u e_2e_3\cdot\psi_2+se_1\cdot\psi_2=0.
\]
Using $e_2e_3\cdot=-e_1\cdot$, we obtain
\[
(-4u+s)e_1\cdot\psi_2=0.
\]
Since $\alpha>0$ and $\beta>0$,
\[
-4u+s=4\alpha+(-3\alpha+\beta)=\alpha+\beta>0,
\]
and hence $-4u+s\neq0$. Therefore, $\psi_2=0$, and consequently $\psi_1=0$. Thus, no non-trivial solution exists.

\subsubsection{$\mathfrak{e}(1,1)$}
For $\mathfrak e(1,1)$, we have
\[
\alpha>0,\qquad \beta<0, \qquad \gamma=0.
\]
Hence
\[
r=-3\alpha-3\beta,\qquad s=-3\alpha+\beta,\qquad u=-\alpha,\qquad v=\alpha-\beta.
\]
\begin{enumerate}
\renewcommand{\labelenumi}{(\arabic{enumi})}
\item Suppose that $r\neq0$. The first equation of $(\ref{eq: unimodular RS system})$ gives
\[
\psi_1=\dfrac{4\alpha}{r}e_3\cdot\psi_2.
\]
Substituting this into the third equation of $(\ref{eq: unimodular RS system})$ yields
\[
4u e_2e_3\cdot\psi_2+se_1\cdot\psi_2=0.
\]
Therefore,
\[
(-4u+s)e_1\cdot\psi_2=0.
\]
Since $r\neq0$,
\[
-4u+s=4\alpha+(-3\alpha+\beta)=\alpha+\beta=-\dfrac{1}{3}r\neq0.
\]
Hence $\psi_2=0$, and consequently $\psi_1=0$. Thus, no non-trivial solution exists in this case.
\item Suppose that $r=0$. Then $\beta=-\alpha$. Hence $s=-4\alpha$ and $v=2\alpha$, so $(\ref{eq: unimodular RS system})$ becomes
\[
\begin{cases}
-4\alpha e_3\cdot\psi_2=0, \\
8\alpha e_3\cdot\psi_1-4\alpha \psi_2=0, \\
-4\alpha e_1\cdot\psi_2=0.
\end{cases}
\]
The first equation gives $\psi_2=0$, and the second then gives $\psi_1=0$. Thus, no non-trivial solution exists in this case either.
\end{enumerate}

\subsubsection{$\mathfrak{sl}(2,\mathbb{R})$}
For $\mathfrak{sl}(2,\mathbb R)$, we have
\[
\alpha>0,\qquad \beta>0, \qquad \gamma<0.
\]
Thus $r=-3\alpha-3\beta+\gamma<0$, and in particular $r\neq0$. The first equation of $(\ref{eq: unimodular RS system})$ can therefore be rewritten as
\[
\psi_1=-\dfrac{4u}{r}e_3\cdot\psi_2.
\]
Substituting this into the third equation of $(\ref{eq: unimodular RS system})$, we obtain
\[
4ue_2e_3\cdot\psi_2+se_1\cdot\psi_2=0.
\]
Therefore,
\[
(-4u+s)e_1\cdot\psi_2=0.
\]
Since $\alpha>0$, $\beta>0$, and $\gamma<0$,
\[
-4u+s=-4(-\alpha+\gamma)+(-3\alpha+\beta-3\gamma)=\alpha+\beta-7\gamma>0,
\]
and hence $-4u+s\neq0$. Therefore, $\psi_2=0$. Substituting this into the first equation of $(\ref{eq: unimodular RS system})$ and using $r\neq0$, we obtain $\psi_1=0$. Thus, no non-trivial solution exists.

\subsubsection{$\mathfrak{su}(2)$} \label{subsubsec: su(2)} 
For $\mathfrak{su}(2)$, we have
\[
\alpha>0,\qquad \beta>0, \qquad \gamma>0.
\]
\begin{enumerate}
\renewcommand{\labelenumi}{(\arabic{enumi})}
\item Suppose that $r\neq0$. The first equation of $(\ref{eq: unimodular RS system})$ gives
\[
\psi_1=-\dfrac{4u}{r}e_3\cdot \psi_2.
\] 
Substituting this into the second equation of $(\ref{eq: unimodular RS system})$, we obtain
\[
\left(s+\dfrac{16uv}{r}\right)\psi_2=0.
\]
Thus, a necessary condition for the existence of a non-trivial solution is $sr+16uv=0$. On the other hand, substituting into the third equation of $(\ref{eq: unimodular RS system})$ gives
\[
4ue_2e_3\cdot\psi_2+se_1\cdot\psi_2=0.
\]
Therefore,
\[
(-4u+s)e_1\cdot\psi_2=0,
\]
and hence a necessary condition for the existence of a non-trivial solution is $s=4u$. That is,
\[
\alpha+\beta=7\gamma.
\]
Furthermore, substituting $s=4u$ into $sr+16uv=0$ gives $4u(r+4v)=0$. Hence
\[
u=0 \qquad \text{or} \qquad r+4v=0.
\]
\begin{enumerate}
\renewcommand{\labelenumii}{(\roman{enumii})}
\item If $u=0$, then $\gamma=\alpha$. The relation $\alpha+\beta=7\gamma$ then implies $\beta=6\alpha$. In this case $\psi_1=0$, while $\psi_2$ is arbitrary, so non-trivial solutions exist.
\item If $r+4v=0$, then
\[
-3\alpha-3\beta+\gamma+4(\alpha-\beta)=0,
\]
which is equivalent to
\[
\alpha+\gamma=7\beta.
\]
Together with $\alpha+\beta=7\gamma$, this yields $\alpha=6\gamma$ and $\beta=\gamma$. In this case $\psi_2$ is arbitrary and $\psi_1=-e_3\cdot\psi_2$, so non-trivial solutions exist.
\end{enumerate}
\item Suppose that $r=0$. Then $\gamma=3\alpha+3\beta$. The first equation of $(\ref{eq: unimodular RS system})$ becomes $4ue_3\cdot\psi_2=0$. Since
\[
u=-\alpha+\gamma=2\alpha+3\beta>0,
\]
we obtain $\psi_2=0$. The second equation of $(\ref{eq: unimodular RS system})$ is $4ve_3\cdot\psi_1=0$, and the existence of a non-trivial solution therefore requires $v=0$. Thus $\alpha=\beta$ and $\gamma=3\alpha+3\beta=6\alpha$. In this case $\psi_2=0$, while $\psi_1$ is arbitrary, so non-trivial solutions exist.
\end{enumerate}
Combining the above cases, we conclude that non-trivial left-invariant Rarita-Schwinger fields exist if and only if 
\[
(\alpha,\beta,\gamma)=(t,t,6t), \quad (t,6t,t), \quad \text{or} \quad (6t,t,t), \quad t>0.
\]
In each of these cases, the space of left-invariant Rarita-Schwinger fields has complex dimension $2$. The corresponding metrics will be described in Subsection \ref{subsec: Berger sphere}.

\subsection{Left-invariant Rarita-Schwinger fields on 3-dimensional non-unimodular Lie groups}
The purpose of this subsection is to classify left-invariant Rarita-Schwinger fields on $3$-dimensional non-unimodular Lie groups. We first recall the standard form of a three-dimensional non-unimodular metric Lie algebra needed for the computation; see \cite[p.~321]{Milnor} and \cite[Section~2.5]{MeeksPerez} for a more detailed discussion.

Let $(\mathfrak g,\langle\cdot,\cdot\rangle)$ be a $3$-dimensional non-unimodular metric Lie algebra. Then
\[
\mathfrak{u}\coloneqq \left\{X\in\mathfrak{g}\mid \tr(\ad_X)=0\right\}
\]
is a $2$-dimensional abelian ideal of $\mathfrak g$, called the unimodular kernel. Let $(e_1,e_2)$ be an orthonormal basis of $\mathfrak u$ and let $e_3$ be a unit vector in $\mathfrak u^\perp$. Then, for some real $2\times2$ matrix $A$,
\[
[e_1,e_2]=0,\qquad \ad_{e_3}|_{\mathfrak{u}}=A.
\]

By the standard form for $3$-dimensional non-unimodular metric Lie algebras \cite[p.321]{Milnor} (see also \cite[Section 2.5]{MeeksPerez}), after rescaling the metric by a positive constant if necessary and suitably changing the orthonormal basis of $\mathfrak u$, we may assume that
\[
A=\begin{pmatrix}
1+a&-(1-a)b\\
(1+a)b&1-a
\end{pmatrix},\qquad a,b\geq0.
\]
Consequently,
\begin{equation} \label{eq: orthogonal basis for non-unimodular lie alg}
\begin{split}
[e_1,e_2]&=0,\\
[e_3,e_1]&=(1+a)e_1+(1+a)be_2, \\
[e_3,e_2]&=-(1-a)be_1+(1-a)e_2.
\end{split}
\end{equation}

Set $c\coloneqq\det A=(1-a^2)(1+b^2)$. If $a=b=0$, equivalently $A=\Id$, then the corresponding Lie algebra is
\[
\mathbb{R}^2\rtimes_{\Id}\mathbb{R}.
\]
On the other hand, if $(a,b)\neq(0,0)$, then $A$ is similar over $\mathbb R$ to
\[
D(c)=\begin{pmatrix}
0&-c \\
1&2
\end{pmatrix}
\]
and therefore the corresponding Lie algebra is isomorphic to
\[
\mathfrak{g}(c)\coloneqq\mathbb{R}^2\rtimes_{D(c)}\mathbb{R}.
\]
Thus, to study $3$-dimensional non-unimodular metric Lie algebras, it is enough to consider the following two cases:
\[
\mathbb{R}^2\rtimes_{\Id}\mathbb{R}, \qquad \mathfrak{g}(c).
\]

Although the cases $(a,b)=(0,0)$ and $(a,b)\neq(0,0)$ have been distinguished above, the following computation applies uniformly for all $a,b\geq 0$. With the above normalization, choose an orthonormal basis $(e_1,e_2,e_3)$ satisfying $(\ref{eq: orthogonal basis for non-unimodular lie alg})$. By the Maurer-Cartan formula $(\ref{eq: maurer-cartan})$,
\[
de_1=(1+a)e_1\wedge e_3-(1-a)be_2\wedge e_3, \qquad de_2=(1+a)be_1\wedge e_3+(1-a)e_2\wedge e_3, \qquad de_3=0.
\]
Thus, in the notation of $(\ref{eq: exterior derivative on 3-dim lie alg})$,
\[
a_{13}=1+a, \qquad a_{23}=-(1-a)b, \qquad b_{13}=(1+a)b, \qquad b_{23}=1-a,
\]
and all other components vanish. Substituting these expressions into $(\ref{eq: twisted Dirac for left-inv spin-3/2 on 3-dim})$, we obtain
\begin{align*}
-4(D_{TM}\Psi)_1=&\left(-2(1-2a)b+4(2+a)e_3\cdot\right)\psi_1+\left(-4(1+a)+4(1+a)be_3\cdot\right)\psi_2, \\
-4(D_{TM}\Psi)_2=&\left(4(1-a)-4(1-a)be_3\cdot\right)\psi_1+\left(-2(1+2a)b+4(2-a)e_3\cdot\right)\psi_2, \\
-4(D_{TM}\Psi)_3=&\left(2(1-2a)be_2\cdot-4(2+a)e_1\cdot\right)\psi_1+\left(-4(2-a)e_2\cdot-2(1+2a)be_1\cdot\right)\psi_2.
\end{align*}
Set $A_1=-2(1-2a)bI+4(2+a)e_3\cdot$, $B_1=-4(1+a)I+4(1+a)be_3\cdot$, $A_2=4(1-a)I-4(1-a)be_3\cdot$, $B_2=-2(1+2a)bI+4(2-a)e_3\cdot$, $A_3=2(1-2a)be_2\cdot-4(2+a)e_1\cdot$, and $B_3=-4(2-a)e_2\cdot-2(1+2a)be_1\cdot$. Then the left-invariant Rarita-Schwinger equation reduces to
\begin{equation}\label{eq: non-unimodular RS system}
\begin{cases}
A_1\psi_1+B_1\psi_2=0,\\
A_2\psi_1+B_2\psi_2=0,\\
A_3\psi_1+B_3\psi_2=0.
\end{cases}
\end{equation}
We now determine whether this system admits a non-trivial solution. By $(\ref{eq: determinant of xI plus ye3})$,
\[
\det A_1=\left\{-2(1-2a)b\right\}^2+\left\{4(2+a)\right\}^2 \geq \left\{4(2+a)\right\}^2>0,
\]
and hence $A_1$ is invertible, with
\begin{align*}
A_1^{-1}=&\dfrac{1}{\det A_1}\left(-2(1-2a)b I-4(2+a)e_3\cdot\right).
\end{align*}
The first equation of $(\ref{eq: non-unimodular RS system})$ gives
\[
\psi_1=-A_1^{-1}B_1\psi_2.
\]
Substituting this into the second equation of $(\ref{eq: non-unimodular RS system})$, we obtain
\[
(B_2-A_2A_1^{-1}B_1)\psi_2=0.
\]
We examine whether the matrix $B_2-A_2A_1^{-1}B_1$ is invertible. First,
\begin{align*}
A_1^{-1}B_1=&\dfrac{1}{\det A_1}\left(-2(1-2a)b I-4(2+a)e_3\cdot \right)\left(-4(1+a)I+4(1+a)be_3\cdot\right) \\
=&\dfrac{8(1+a)}{\det A_1}\left(5b I+\left(-(1-2a)b^2+2(2+a)\right)e_3\cdot \right).
\end{align*}
Therefore,
\begin{align*}
A_2A_1^{-1}B_1=&\dfrac{8(1+a)}{\det A_1}\left(4(1-a)I-4(1-a)be_3\cdot\right)\left(5b I+\left(-(1-2a)b^2+2(2+a)\right)e_3\cdot \right) \\
=&\dfrac{32(1-a^2)}{\det A_1}\left(\left(-(1-2a)b^3+(9+2a)b\right)I+\left(-2(3-a)b^2+2(2+a)\right)e_3\cdot\right).
\end{align*}
It follows that
\begin{align*}
&B_2-A_2A_1^{-1}B_1 \\
=&-2(1+2a)bI+4(2-a)e_3\cdot \\
&-\dfrac{32(1-a^2)}{\det A_1}\left(\left(-(1-2a)b^3+(9+2a)b\right)I+\left(-2(3-a)b^2+2(2+a)\right)e_3\cdot\right) \\
=&\dfrac{1}{\det A_1}\left[-8(1-2a)^2(1+2a)b^3I-32(2+a)^2(1+2a)bI \right. \\
&\left. +32(1-a^2)(1-2a)b^3I-32(1-a^2)(9+2a)bI \right. \\
&\left. +16(1-2a)^2(2-a)b^2e_3\cdot + 64(2+a)^2(2-a)e_3\cdot \right. \\
&\left. +32(1-a^2)(-6+2a)b^2e_3\cdot - 64(1-a^2)(2+a)e_3\cdot\right] \\
=&\dfrac{8b}{\det A_1}\left[ (3-6a)b^2-(52+56a) \right]I+\dfrac{16}{\det A_1}\left[ (14-13a)b^2+(24+12a) \right]e_3\cdot.
\end{align*}
Consequently, by $(\ref{eq: determinant of xI plus ye3})$,
\begin{align*}
&\det(B_2-A_2A_1^{-1}B_1) \\
=&\left(\dfrac{8b}{\det A_1}\left[ (3-6a)b^2-(52+56a) \right]\right)^2+\left(\dfrac{16}{\det A_1}\left[ (14-13a)b^2+(24+12a) \right]\right)^2 \\
=&\frac{16(b^2+36)(9b^2+4)}{\det A_1}>0.
\end{align*}
Thus $B_2-A_2A_1^{-1}B_1$ is invertible. Hence $(B_2-A_2A_1^{-1}B_1)\psi_2=0$ implies $\psi_2=0$, and no non-trivial solution exists.

\subsection{The Berger sphere example} \label{subsec: Berger sphere}
In the $\mathfrak{su}(2)$ case, we showed that the Milnor constants for which non-trivial left-invariant Rarita-Schwinger fields exist are
\[
(\alpha,\beta,\gamma)=(t,t,6t),\quad (t,6t,t),\quad (6t,t,t), \qquad t>0.
\]
In this subsection, we describe the corresponding left-invariant metrics on $\SU(2)\cong S^3$.

Let $(e_1,e_2,e_3)$ be an orthonormal Milnor frame satisfying
\[
[e_1,e_2]=\alpha e_3,\qquad [e_2,e_3]=\beta e_1,\qquad [e_3,e_1]=\gamma e_2.
\]
Set
\[
x_1^2=\frac{4}{\alpha\gamma},\qquad x_2^2=\frac{4}{\alpha\beta},\qquad x_3^2=\frac{4}{\beta\gamma},
\]
and define
\[
E_i=x_i e_i\qquad (i=1,2,3).
\]
Then
\[
[E_1,E_2]=2E_3,\qquad [E_2,E_3]=2E_1,\qquad [E_3,E_1]=2E_2,
\]
so $(E_1,E_2,E_3)$ is a standard basis of $\mathfrak{su}(2)$. Since $(e_1,e_2,e_3)$ is orthonormal with respect to $g$, the matrix representing $g$ with respect to this standard basis is
\[
\operatorname{diag}(x_1^2,x_2^2,x_3^2).
\]

For example, if $(\alpha,\beta,\gamma)=(t,t,6t)$, then
\[
g=\frac{2}{3t^2}\operatorname{diag}(1,6,1).
\]
Similarly, in the other two cases we obtain, respectively,
\[
g=\frac{2}{3t^2}\operatorname{diag}(6,1,1), \qquad g=\frac{2}{3t^2}\operatorname{diag}(1,1,6).
\]
These metrics are transformed into one another by cyclic permutations of the standard basis, that is, by automorphisms of $\mathfrak{su}(2)$.

A left-invariant metric represented with respect to the standard basis by
\[
\operatorname{diag}(\varepsilon,1,1), \qquad \varepsilon>0,
\]
is called a Berger metric \cite{GadeaOubina}. Therefore, up to an automorphism of the Lie group and multiplication by a positive constant, the three metrics obtained above are represented by the Berger metric
\[
\operatorname{diag}(6,1,1).
\]
This metric is, of course, not Einstein.

Combining this with the abelian case discussed in Remark \ref{rem: abel lie algebra}, we obtain the second main theorem.
\let\temp\thetheorem
\renewcommand{\thetheorem}{\ref{thmB}}

\begin{theorem}
Let $(G,g)$ be a connected and simply connected $3$-dimensional Lie group endowed with a left-invariant Riemannian metric. If $G$ is abelian, then every left-invariant spin-$3/2$ spinor is parallel and hence defines a Rarita-Schwinger field. If $G$ is non-abelian, then a non-trivial left-invariant Rarita-Schwinger field exists if and only if $G\cong\SU(2)$ and, up to an automorphism of the Lie group and multiplication by a positive constant, $g$ is represented with respect to the standard basis by
\[
\begin{pmatrix}
6&0&0\\
0&1&0\\
0&0&1
\end{pmatrix}.
\]
In this case, the space of left-invariant Rarita-Schwinger fields has complex dimension $2$. Moreover, $g$ is a Berger metric on $\SU(2)\cong S^3$.
\end{theorem}

\let\thetheorem\temp
\addtocounter{theorem}{-1}

\begin{remark}
Although we restrict our main results to connected and simply connected Lie groups, the computation of left-invariant Rarita-Schwinger fields depends only on the corresponding metric Lie algebra. Therefore, the same Lie-algebraic computation applies to connected Lie groups that are not simply connected, provided that they are equipped with the spin structure induced by a left-invariant framing. In particular, a connected Lie group with Lie algebra $\mathfrak{su}(2)$ is isomorphic either to $\SU(2)$ or to $\SO(3)\cong \SU(2)/\{\pm I\}$. Hence the exceptional metric also yields a complex $2$-dimensional space of left-invariant Rarita-Schwinger fields on $\SO(3)$ with its spin structure induced by a left-invariant framing. The exceptional metric on $\mathbb{R}P^3\cong\SO(3)$ is the quotient of the Berger metric on $S^3\cong\SU(2)$.
\end{remark}

It is known from the work of Y. Homma and U. Semmelmann \cite{HommaSemmelmann} that the standard $3$-sphere with the round metric admits no Rarita-Schwinger fields. On the other hand, for the Berger sphere obtained here, the space of left-invariant Rarita-Schwinger fields has complex dimension $2$. This is the second known example in which the existence of
Rarita-Schwinger fields on the same manifold depends on the choice of metric. The first example, on $S^3\times S^3$, was obtained by S. Ohno and T. Tomihisa \cite{OhnoTomihisa}.

There are also known non-Einstein examples carrying Rarita-Schwinger fields in higher dimensions. In dimension $8$, harmonic $\Sp(1)\Sp(2)$-structures give rise to non-trivial Rarita-Schwinger fields; see F. Witt \cite{Witt}. Several compact non-Einstein examples of such structures are known. D. Conti and T. B. Madsen \cite{ContiMadsen} constructed non-Einstein harmonic $\Sp(2)\Sp(1)$-structures on $8$-dimensional nilmanifolds, and D. Conti, T. B. Madsen, and S. Salamon \cite{ContiMadsenSalamon} constructed non-Einstein harmonic $\Sp(2)\Sp(1)$-structures on $\G_2/\SO(4)$. Our result provides a low-dimensional compact homogeneous non-Einstein example carrying non-trivial left-invariant Rarita-Schwinger fields.

\section*{Acknowledgements}
The author would like to express his sincere gratitude to Professor Yasushi Homma and Natsuki Imada for valuable discussions and helpful comments. This work was supported by JSPS KAKENHI Grant Number JP24K06721.

\end{document}